\documentclass[11pt]{article}
\usepackage[margin=1.1in]{geometry}
\usepackage{amsmath,amssymb,amsthm,mathtools}
\usepackage{enumitem}
\usepackage{hyperref}
\usepackage[nameinlink,capitalize,noabbrev]{cleveref}
\usepackage{microtype}

\hypersetup{colorlinks=true,linkcolor=blue,citecolor=blue,urlcolor=blue}

\newtheorem{theorem}{Theorem}[section]

\newtheorem{lemma}[theorem]{Lemma}
\newtheorem{corollary}[theorem]{Corollary}

\newtheorem{remark}[theorem]{Remark}

\crefname{theorem}{Theorem}{Theorems}
\crefname{proposition}{Proposition}{Propositions}
\crefname{lemma}{Lemma}{Lemmas}
\crefname{corollary}{Corollary}{Corollaries}
\crefname{definition}{Definition}{Definitions}
\crefname{remark}{Remark}{Remarks}

\DeclareMathOperator{\Prob}{\mathbb P}
\DeclareMathOperator{\E}{\mathbb E}
\newcommand{\Ind}{\mathbf 1}
\newcommand{\dd}{\mathrm d}

\title{The Ballot Event for Two-Player Coupon Collection: A Renewal--Catalan Asymptotic}
\author{Christopher D. Long\\
Headlamp Software\\
\texttt{galizur@gmail.com}}
\date{}

\begin{document}
\maketitle

\begin{abstract}
We study the two-player coupon-collector competition in which two independent collectors draw one coupon each per round from a set of $d$ equally likely coupon types.  Myers and Wilf gave finite formulae for several two-player events and explicitly left open the ballot-type problem of finding the probability that the ultimate winner was never behind.  We prove that this probability satisfies
\[
        b_d \sim \frac{2}{d}, \qquad d\to\infty .
\]
The proof uses a renewal decomposition at the tie boundary.  The first one-sided tie-break has an explicit entrance distribution; its level, scaled by $d^{1/2}$, converges to a Rayleigh law; and, after the break, the leader's survival probability is governed by a Catalan, or gambler's-ruin, harmonic.  The main estimate shows that the accumulated defect of this comparison harmonic in the exact simultaneous-round chain is negligible.
\end{abstract}

\section{Introduction}

The classical coupon collector asks for the time required until a single collector has seen all $d$ coupon types.  Myers and Wilf \cite{MW} studied several refined variants, including two simultaneous collectors.  In their two-player model, in every round collector $A$ and collector $B$ independently draw one coupon from $\{1,\dots,d\}$, uniformly and with replacement.  They solved, among other things, the event that the two collectors complete their collections in the same round, and a related one-sided problem in which the first player to pull ahead never loses the lead.  They also considered the more natural ballot event: the ultimate winner was never behind.  That latter problem was left unresolved; Myers and Wilf described a decomposition into tails, frames, and ribbons, but no evaluation of this probability was obtained \cite[Sec.~2.9]{MW}.

The purpose of this note is to give an asymptotic solution of the Myers--Wilf winner-never-behind problem.  We do not attempt to produce an exact finite formula for each fixed value of $d$.

\begin{theorem}[Winner-never-behind asymptotic]\label{thm:main}
Let $b_d$ be the probability that, in the simultaneous two-player coupon-collector competition with $d$ coupon types, the ultimate winner was never behind the other collector in number of distinct coupon types collected.  Then
\[
        b_d \sim \frac{2}{d}.
\]
\end{theorem}

We use the phrase \emph{terminal-condition flux argument} in a modest, self-contained sense.  It means that the event is decomposed according to first entry through a relevant terminal or boundary set, that the asymptotic distribution of this entry point is identified, and that the conditional probability of final success from the entry state is then estimated.  In the present problem the relevant boundary is the tie set $g=0$; the first one-sided tie-break supplies the entrance mass; and the post-break survival probability is controlled by a Catalan comparison harmonic.

The proof is not a direct use of the Gessel--Viennot machinery \cite{GV}.  Although the problem is reminiscent of nonintersecting lattice paths, the simultaneous-round dynamics do not give a fixed-step path model after projection to the two collection counts.  Starting from a tie, the process typically remains tied for a long diagonal segment.  The first one-sided tie-break occurs when each collector has seen order $d^{1/2}$ distinct coupon types.  Conditional on such a break, the leader begins with lead one.  The remaining problem is then a lead excursion in a shrinking wedge.  Its leading comparison function is the Catalan harmonic for the one-dimensional comparison walk,
\[
        H(s,g)=\frac{g+1}{d-s+1}.
\]
In the exact simultaneous-round chain this function is not perfectly harmonic; it is slightly subharmonic, in the Markov-chain sense that its expected one-step change is nonnegative.  The main analytic step is to prove that the accumulated defect of $H$ under the exact transition operator is negligible in the first-break window.

\section{The one-sided event and changed-state chain}

Let $C_A(n)$ and $C_B(n)$ denote the numbers of distinct coupons seen by collectors $A$ and $B$ after $n$ simultaneous rounds.  Let
\[
        T_A:=\inf\{n:C_A(n)=d\},\qquad
        T_B:=\inf\{n:C_B(n)=d\}.
\]
The event that $A$ wins and is never behind is
\[
        \mathcal E_A:=\{T_A<T_B\}\cap\{C_A(n)\ge C_B(n)\text{ for every }n\ge0\}.
\]
Terminal ties are not wins for either player.  The events that $A$ wins and was never behind and that $B$ wins and was never behind are disjoint, are exchanged by symmetry, and have union equal to the event defining $b_d$.  Hence
\[
        b_d=2\Prob(\mathcal E_A).
\]
We write
\[
        E_d:=\Prob(\mathcal E_A).
\]
Thus the theorem is equivalent to
\[
        E_d\sim \frac{1}{d}.
\]

It is useful to use coordinates adapted to the tie boundary.  Let
\[
        s=C_B,
        \qquad
        g=C_A-C_B.
\]
Thus $s$ is the number of distinct coupons already collected by collector $B$, and $g$ is $A$'s lead over $B$.  The safe region is
\[
        \mathcal W_d:=\{(s,g):0\le s\le d,
        \ 0\le g\le d-s\}.
\]
If $B$ has collected $s$ types and $A$ has lead $g$, then $B$ is missing
\[
        m=d-s
\]
coupon types, while $A$ is missing $m-g$ types.

Self-loops, in which neither distinct-coupon count changes, will be erased.  We record the elementary reduction explicitly.

\begin{lemma}[Self-loop erasure and finite absorption]\label{lem:self-loop-erasure}
Consider the original simultaneous-round chain, stopped when it first reaches the success boundary for $A$, the unsafe boundary $g=-1$, or the simultaneous terminal point $(d,0)$.  If one deletes all rounds in which neither distinct-coupon count changes, then the induced sequence of nonself transitions is a Markov chain whose transition probabilities are proportional to the three weights displayed below.  Consequently all hitting probabilities of these absorbing sets are unchanged by self-loop erasure.

Moreover, the changed-state chain is absorbed after at most $2d-(2s+g)$ nonself transitions when started from an interior state $(s,g)$.  In particular, all stopping times below are almost surely finite and bounded by $2d$.
\end{lemma}

\begin{proof}
At state $(s,g)$, collector $A$ has $s+g$ distinct coupons and is missing $m-g=d-s-g$ coupons, while collector $B$ has $s$ distinct coupons and is missing $m=d-s$ coupons.  In one simultaneous round, the probabilities of the three nonself events are
\[
        \frac{(m-g)m}{d^2},\qquad
        \frac{(m-g)s}{d^2},\qquad
        \frac{(s+g)m}{d^2},
\]
corresponding respectively to both collectors getting new coupons, only $A$ getting a new coupon, and only $B$ getting a new coupon.  Conditional on the event that a nonself transition occurs, the factor $d^{-2}$ cancels, giving probabilities proportional to the displayed weights.  Deleting a geometrically distributed number of self-loops before each nonself transition therefore does not change the law of the next nonself state, nor any hitting probability of an absorbing set.

Finally, at every nonself transition the quantity
\[
        C_A+C_B=2s+g
\]
increases by at least one.  Since $C_A+C_B\le 2d$ before or at absorption, no path can contain more than $2d-(2s+g)$ nonself transitions from the starting state $(s,g)$.
\end{proof}

From an interior state $0\le g<d-s$, the changed-state transitions are as follows:
\[
\begin{array}{c|c|c}
\text{event} & \text{new state} & \text{weight}\\
\hline
\text{both get new coupons} & (s+1,g) & (d-s-g)(d-s)\\[0.3em]
A\text{ alone gets a new coupon} & (s,g+1) & (d-s-g)s\\[0.3em]
B\text{ alone gets a new coupon} & (s+1,g-1) & (s+g)(d-s).
\end{array}
\]
The last transition is killing if $g=0$.  The success boundary for $A$ is
\[
        g=d-s\ge1,
\]
because then $A$ has completed its collection and $B$ has not.  The simultaneous terminal point is
\[
        (s,g)=(d,0),
\]
which is not a win for either player.

Let $U_d(s,g)$ be the probability that $A$ eventually wins and never falls behind, starting from $(s,g)$.  Then
\[
        E_d=U_d(0,0).
\]
The boundary conditions are
\[
        U_d(s,d-s)=1\quad(0\le s<d),
        \qquad
        U_d(d,0)=0.
\]
In the interior, with $m=d-s$, put
\[
        D(s,g):=(m-g)m+(m-g)s+(s+g)m.
\]
Then
\begin{equation}\label{eq:U-recursion}
U_d(s,g)=
\frac{(m-g)mU_d(s+1,g)+(m-g)sU_d(s,g+1)+(s+g)mU_d(s+1,g-1)}{D(s,g)},
\end{equation}
where $U_d(s+1,-1)$ is interpreted as $0$.

\section{Tie skeleton and first-break distribution}

Define
\[
        e_m:=U_d(d-m,0),
        \qquad
        w_m:=U_d(d-m,1).
\]
Here $m$ is the number of coupon types still missing for $B$; in particular, $e_d=E_d$.  At a tie with $m$ coupons still missing for both collectors, after self-loops are erased, there are three possibilities:
\[
\begin{array}{c|c|c}
\text{event} & \text{new state} & \text{conditional probability}\\
\hline
\text{both get new coupons} & \text{tie with }m-1\text{ missing} & \frac{m}{2d-m}\\[0.5em]
A\text{ alone gets a new coupon} & \text{lead one, }m\text{ missing for }B & \frac{d-m}{2d-m}\\[0.5em]
B\text{ alone gets a new coupon} & \text{failure for }A & \frac{d-m}{2d-m}.
\end{array}
\]
These conditional probabilities are the tie-state weights from the changed-state transition table, divided by their sum.  Therefore
\begin{equation}\label{eq:tie-recursion}
        e_m=\frac{m}{2d-m}e_{m-1}+\frac{d-m}{2d-m}w_m.
\end{equation}
We now unroll this renewal recursion along the tie boundary.  Write
$r=d-m$ for the tied level, that is, the number of distinct coupon types seen
by each collector while the process is still tied.  At tied level $r$, the
probability that the next nonself transition preserves the tie is
\[
        \frac{d-r}{d+r},
\]
and the probability that $A$ alone makes the first one-sided move is
\[
        \frac{r}{d+r}.
\]
Thus, for $A$ to be the first player to break the tie at level $k$, the chain
must preserve the tie at levels $0,1,\dots,k-1$ and then break in favor of
$A$ at level $k$.  The level-$0$ tie-preserving factor is equal to one, so it
is omitted from the product.  This gives the exact tie-skeleton formula
\begin{equation}\label{eq:tie-skeleton}
        E_d=
        \sum_{k=1}^{d-1}
        \left(\prod_{r=1}^{k-1}\frac{d-r}{d+r}\right)
        \frac{k}{d+k}\, w_{d-k}.
\end{equation}

Let
\begin{equation}\label{eq:pi-def}
        \pi_{d,k}:=
        \left(\prod_{r=1}^{k-1}\frac{d-r}{d+r}\right)
        \frac{2k}{d+k},
        \qquad 1\le k\le d-1.
\end{equation}
Thus $\pi_{d,k}$ is the probability that the first one-sided tie-break occurs
at level $k$, regardless of which player breaks it.  Equivalently, the
coefficient in \eqref{eq:tie-skeleton} is $\pi_{d,k}/2$, and hence
\begin{equation}\label{eq:tie-expectation}
        E_d=\frac{1}{2}\sum_{k=1}^{d-1}\pi_{d,k}w_{d-k}.
\end{equation}
The remaining tie-stage mass is
\[
        \rho_d=\prod_{r=1}^{d-1}\frac{d-r}{d+r},
\]
corresponding to preservation of the tie all the way to simultaneous
completion.  This is the tie-stage hazard decomposition: starting from a tie,
the bounded tie-stage evolution ends either with a first one-sided break at a
unique level $1\le k\le d-1$, in favor of one of the two players, or with no
one-sided break before simultaneous completion.  Therefore
\[
        \sum_{k=1}^{d-1}\pi_{d,k}+\rho_d=1.
\]

\begin{lemma}[First-break localization and tails]\label{lem:first-break}
Let $K_d$ be the extended first-break level defined by
\[
        \Prob(K_d=k)=\pi_{d,k}\quad(1\le k\le d-1),
        \qquad
        \Prob(K_d=d)=\rho_d,
\]
where
\[
        \rho_d:=\prod_{r=1}^{d-1}\frac{d-r}{d+r}
\]
is the probability that no one-sided tie-break occurs before simultaneous completion.  Thus the value $K_d=d$ records this exceptional terminal event.  Then
\[
        \frac{K_d}{\sqrt d}\Rightarrow X,
        \qquad
        \Prob(X\in \dd x)=2x e^{-x^2}\Ind_{\{x>0\}}\dd x .
\]
In particular,
\[
        K_d\to\infty,
        \qquad
        \frac{K_d}{d}\to0
\]
in probability.  More precisely, for $k=o(d)$,
\[
        \prod_{r=1}^{k-1}\frac{d-r}{d+r}
        =\exp\left(-\frac{k(k-1)}{d}+O\left(\frac{k^4}{d^3}\right)\right),
\]
and there are absolute constants $C,c>0$ such that, for every $R\ge1$ and all sufficiently large $d$,
\begin{equation}\label{eq:first-break-tail}
        \sum_{k\ge R\sqrt{d\log d}}\pi_{d,k}
        \le C d^{-cR^2}+e^{-cd}.
\end{equation}
More generally, for every $M\ge1$ and all sufficiently large $d$,
\begin{equation}\label{eq:first-break-rayleigh-tail}
        \sum_{k\ge M\sqrt d}\pi_{d,k}
        \le C e^{-cM^2}+e^{-cd}.
\end{equation}
In particular, for $R$ large enough the tail in
\eqref{eq:first-break-tail} is $o(d^{-1})$.
\end{lemma}

\begin{proof}
For $k=o(d)$, uniformly in that range,
\[
\log\prod_{r=1}^{k-1}\frac{d-r}{d+r}
=
\sum_{r=1}^{k-1}\log\left(\frac{1-r/d}{1+r/d}\right).
\]
Since
\[
\log\left(\frac{1-x}{1+x}\right)=-2x+O(x^3)
\]
uniformly for $0\le x\le 1/2$, the asserted expansion follows.  Consequently, for each fixed $M<\infty$, uniformly for $1\le k\le M\sqrt d$,
\[
        \pi_{d,k}
        =\frac{2k}{d}
        \exp\left(-\frac{k^2}{d}+o(1)\right).
\]
Thus, for every $0\le u<v<\infty$, Riemann-sum convergence gives
\[
        \sum_{u\sqrt d<k\le v\sqrt d}\pi_{d,k}
        \longrightarrow
        \int_u^v 2x e^{-x^2}\dd x .
\]

For the quantitative upper tail, first suppose $k\le d/2$.  Since
\[
        \log\left(\frac{1-x}{1+x}\right)\le -2x
        \qquad (0\le x<1),
\]
we have
\[
        \prod_{r=1}^{k-1}\frac{d-r}{d+r}
        \le \exp\left(-\frac{k(k-1)}{d}\right).
\]
Also $2k/(d+k)\le 2k/d$.  Summing by comparison with the integral of the Rayleigh tail gives, for every $M\ge1$,
\[
        \sum_{M\sqrt d\le k\le d/2}\pi_{d,k}
        \le C e^{-cM^2}.
\]
Taking $M=R\sqrt{\log d}$ gives
\[
        \sum_{R\sqrt{d\log d}\le k\le d/2}\pi_{d,k}
        \le C d^{-cR^2}.
\]
For $k>d/2$, the product has already acquired an exponentially small factor at $\lfloor d/2\rfloor$; hence that part of the tail is $O(e^{-cd})$.  This proves both \eqref{eq:first-break-tail} and \eqref{eq:first-break-rayleigh-tail}.  Finally, the event of no one-sided break before terminal completion has probability
\[
        \prod_{r=1}^{d-1}\frac{d-r}{d+r}
        =\frac{d!(d-1)!}{(2d-1)!}=O(e^{-cd}),
\]
by Stirling's formula.  The bound with threshold $M\sqrt d$ proves tightness of $K_d/\sqrt d$, and the Riemann-sum convergence on compact intervals then proves the displayed Rayleigh convergence.  This convergence implies $K_d\to\infty$ in probability, while tightness on the $\sqrt d$ scale gives $K_d/d\to0$ in probability.  The logarithmic tail bound is exactly \eqref{eq:first-break-tail}.
\end{proof}

\section{The lead-excursion theorem}

The tie-skeleton formula reduces the proof of Theorem~\ref{thm:main} to the asymptotics of $w_m=U_d(d-m,1)$ for $m=d-k$ with $k$ in the first-break window.  In $(s,g)$ coordinates this means $s=k$, $g=1$, with $s=o(d)$.  We prove the needed estimate uniformly down to bounded ages; this avoids any small-$k$ tail loss in the final summation.  Ties remain admissible during this never-behind excursion: the first unsafe state is $g=-1$, not $g=0$.

For the one-dimensional comparison walk in the wedge with absorbing lower boundary at $g=-1$ and upper boundary at $g=d-s$, the usual Catalan, or gambler's-ruin, harmonic is
\begin{equation}\label{eq:H-def}
        H(s,g):=\frac{g+1}{d-s+1}=\frac{g+1}{m+1}.
\end{equation}
We use this as a comparison function for the simultaneous-round chain.
The same formula gives the boundary values
\[
        H(s,-1)=0,
        \qquad
        H(s,d-s)=1\quad(0\le s<d),
        \qquad
        H(d,0)=1.
\]
The value at $(d,0)$ is used only for optional stopping: the simultaneous terminal point is not a win for either player, and its contribution is subtracted explicitly below.

\begin{theorem}[Lead-excursion estimate]\label{thm:lead-excursion}
Fix $g_0\ge0$.  Uniformly for integer sequences $s_0=s_0(d)$ satisfying $0\le s_0<d-g_0$ and $s_0=o(d)$, one has
\begin{equation}\label{eq:lead-estimate}
        U_d(s_0,g_0)=H(s_0,g_0)
        \left(1+O_{g_0}\left(\frac{s_0}{d}+\frac{\log d}{\sqrt d}\right)\right).
\end{equation}
In particular, if $s_0=O(d^{1/2}\log d)$ and $g_0=1$, then
\[
        U_d(s_0,1)=\frac{2}{d-s_0+1}(1+o(1))=\frac{2}{d}(1+o(1)),
\]
where the second $o(1)$ also absorbs the factor $s_0/d$.
\end{theorem}

\begin{remark}
The estimate is stated also for $g_0=0$; this stronger form is a useful by-product and gives the main asymptotic directly from $E_d=U_d(0,0)$.  We nevertheless prove Theorem~\ref{thm:main} through the tie-skeleton formula, because that route exposes the renewal structure at the tie set, the Rayleigh first-break law, and the Catalan lead excursion after the first one-sided break.  The direct $g_0=0$ estimate is analytically shorter, but it obscures the source of the constant.
\end{remark}

The proof occupies the next two sections.  The central point is a Green estimate for the defect of $H$.

\section{Defect of the Catalan comparison harmonic}

Let $P$ denote the changed-state transition operator in the interior of $\mathcal W_d$, with the unsafe transition from $g=0$ to $g=-1$ assigned value $0$.  For a function $F$,
\[
(PF)(s,g)=
\frac{(m-g)mF(s+1,g)+(m-g)sF(s,g+1)+(s+g)mF(s+1,g-1)}{D(s,g)}.
\]
Define
\[
        \Delta(s,g):=(PH)(s,g)-H(s,g).
\]

\begin{lemma}[Exact defect]\label{lem:defect}
For every interior state $0\le g<d-s$, with $m=d-s$,
\begin{equation}\label{eq:Delta-exact}
        \Delta(s,g)=\frac{m-g}{(m+1)D(s,g)}.
\end{equation}
In particular, $H$ is subharmonic in the probabilistic sense $PH\ge H$.  Moreover,
\begin{equation}\label{eq:Delta-simple-bound}
        0\le \Delta(s,g)\le \frac{1}{d(m+1)}.
\end{equation}
\end{lemma}

\begin{proof}
From \eqref{eq:H-def},
\[
H(s+1,g)=\frac{g+1}{m},
\quad
H(s,g+1)=\frac{g+2}{m+1},
\quad
H(s+1,g-1)=\frac{g}{m}.
\]
Substitution in the definition of $P$ gives
\[
(PH)(s,g)=
\frac{(m-g)m\frac{g+1}{m}+(m-g)s\frac{g+2}{m+1}+(s+g)m\frac{g}{m}}{D(s,g)}.
\]
When $g=0$, the last numerator term is zero and corresponds to the convention $H(s+1,-1)=0$.
Subtracting $(g+1)/(m+1)$ and simplifying gives \eqref{eq:Delta-exact}.  Finally,
\[
D(s,g)=(m-g)d+(s+g)m\ge d(m-g),
\]
which implies \eqref{eq:Delta-simple-bound}.
\end{proof}

Let $\tau$ be the first time the changed-state chain hits the success boundary $g=d-s$, the unsafe boundary $g=-1$, or the simultaneous terminal point $(d,0)$.  By Lemma~\ref{lem:self-loop-erasure}, $\tau$ is almost surely finite and bounded by $2d$.

\section{A dyadic Green estimate}

We now prove the main estimate controlling the accumulated defect of $H$.  The estimate rests on two simple facts.  First, by the time the age coordinate has reached $a$, the lead has had order $a^2/d$ chances to change; survival to age $a$ therefore costs order $\sqrt d/a$.  Second, the defect of $H$ at age $s$ is at most $1/(d(d-s+1))$.  A dyadic summation over age then loses only a logarithmic factor.

\begin{lemma}[Dyadic Green estimate]\label{lem:green}
Fix $g_0\ge0$, and let $s_0=s_0(d)$ be an integer sequence satisfying $0\le s_0<d-g_0$ and $s_0=o(d)$.  Then, for all sufficiently large $d$,
\begin{equation}\label{eq:green-estimate}
        \E_{s_0,g_0}\sum_{n<\tau}\Delta(S_n,G_n)
        \le
        C_{g_0}H(s_0,g_0)
        \left(\frac{s_0}{d}+\frac{\log d}{\sqrt d}\right),
\end{equation}
where $C_{g_0}$ depends only on $g_0$.
Consequently,
\[
        \E_{s_0,g_0}\sum_{n<\tau}\Delta(S_n,G_n)=o(H(s_0,g_0)).
\]
\end{lemma}

The proof uses two elementary auxiliary estimates: a survival estimate and an occupation estimate.  We first record the elementary coupling fact used in the survival estimate.

\begin{lemma}[Conditional Bernoulli domination]
\label{lem:bernoulli-domination}
Let $N\ge1$, and let $(\mathcal F_i)_{i=0}^{N}$ be a filtration.  For
$i=1,\dots,N$, let $R_i\in\mathcal F_{i-1}$ be a reachability event and let
$I_i$ be a $\{0,1\}$-valued random variable which is defined on $R_i$ and
measurable with respect to $\mathcal F_i$.  Suppose that, on $R_i$,
\[
        \Prob(I_i=1\mid \mathcal F_{i-1})\ge p_i
\]
for deterministic numbers $0\le p_i\le1$.

Define the completed variables
\[
        \widehat I_i :=
        \begin{cases}
        I_i, & \text{on }R_i,\\
        1,   & \text{on }R_i^c.
        \end{cases}
\]
Let $\xi_1,\dots,\xi_N$ be independent Bernoulli variables with
$\Prob(\xi_i=1)=p_i$.  Then, for every real $t$,
\begin{equation}\label{eq:bernoulli-domination}
        \Prob\left(\sum_{i=1}^{N}\widehat I_i<t\right)
        \le
        \Prob\left(\sum_{i=1}^{N}\xi_i<t\right).
\end{equation}
Consequently, if $A$ is any event such that $A\subseteq\bigcap_{i=1}^{N}R_i$,
then
\begin{equation}\label{eq:bernoulli-domination-reached}
        \Prob\left(A,\sum_{i=1}^{N}I_i<t\right)
        \le
        \Prob\left(\sum_{i=1}^{N}\xi_i<t\right).
\end{equation}
\end{lemma}

\begin{proof}
The completed variables satisfy
\[
        \Prob(\widehat I_i=1\mid \mathcal F_{i-1})\ge p_i
        \qquad(1\le i\le N).
\]
Indeed, on $R_i$ this is the assumed inequality, while on $R_i^c$ we have
$\widehat I_i=1$.

We prove the lower-tail domination by backward induction.  For
$0\le i\le N$, set
\[
        \Phi_i(x):=
        \Prob\left(x+\sum_{j=i+1}^{N}\xi_j<t\right),
\]
with the convention that the empty sum is zero.  The function $\Phi_i(x)$ is
nonincreasing in $x$.  Write
\[
        \Sigma_i:=\widehat I_1+\cdots+\widehat I_i.
\]
Let
\[
        q_i:=\Prob(\widehat I_i=1\mid \mathcal F_{i-1}).
\]
Since $q_i\ge p_i$ and $\Phi_i$ is nonincreasing,
\[
\begin{aligned}
\E\!\left[\Phi_i(\Sigma_{i-1}+\widehat I_i)\mid \mathcal F_{i-1}\right]
&=
q_i\Phi_i(\Sigma_{i-1}+1)+(1-q_i)\Phi_i(\Sigma_{i-1})  \\
&\le
p_i\Phi_i(\Sigma_{i-1}+1)+(1-p_i)\Phi_i(\Sigma_{i-1})  \\
&=
\Phi_{i-1}(\Sigma_{i-1}).
\end{aligned}
\]
Iterating this inequality for $i=N,N-1,\dots,1$ gives
\[
        \Prob\left(\sum_{i=1}^{N}\widehat I_i<t\right)
        =
        \E\,\Phi_N(\Sigma_N)
        \le
        \Phi_0(0)
        =
        \Prob\left(\sum_{i=1}^{N}\xi_i<t\right).
\]
Finally, if $A\subseteq\bigcap_i R_i$, then on $A$ one has
$\widehat I_i=I_i$ for every $i$.  Hence
\[
        \left\{A,\sum_i I_i<t\right\}
        \subseteq
        \left\{\sum_i\widehat I_i<t\right\},
\]
which proves \eqref{eq:bernoulli-domination-reached}.
\end{proof}

\begin{lemma}[Lead-change survival bound]\label{lem:survival}
Fix $g_0\ge0$, and define
\[
        \tau_a:=\inf\{n:S_n\ge a\}.
\]
For every integer $a$ with $\max\{s_0,1\}\le a\le d$ one has the trivial bound
\[
        \Prob_{s_0,g_0}(\tau_a<\tau)\le 1.
\]
Moreover, there is a constant $C_{g_0}$ such that, whenever
\[
        a\ge 4\max\{s_0,\lceil\sqrt d\rceil\},
\]
one has
\begin{equation}\label{eq:survival-bound}
        \Prob_{s_0,g_0}(\tau_a<\tau)
        \le
        C_{g_0}\frac{\sqrt d}{a}.
\end{equation}
\end{lemma}

\begin{proof}
The trivial bound is immediate.  It remains to prove \eqref{eq:survival-bound}; throughout the proof assume $a\ge4\max\{s_0,\lceil\sqrt d\rceil\}$.

Let
\[
        L_a:=
        \sum_{n\ge0}
        \Ind_{\{n<\tau_a\wedge\tau\}}
        \Ind_{\{|G_{n+1}-G_n|=1\}} .
\]
Thus $L_a$ counts all lead-changing transitions whose pre-transition state occurs before $\tau_a\wedge\tau$; in particular, it includes a killing transition from $g=0$ to $g=-1$ when that transition is made before the stopped state.

For an age level $r$, define
\[
        \sigma_r:=\inf\{n:S_n=r\}.
\]
Conditional on the chain reaching age $r$ before absorption, the probability that the next changed-state transition is lead-changing is, at state $(r,g)$,
\[
        \frac{(m-g)r+(r+g)m}{D(r,g)}.
\]
Since
\[
        d\bigl((m-g)r+(r+g)m\bigr)-rD(r,g)
        =m^2(r+g)\ge0,
\]
this probability is at least $r/d$.  Hence, conditional on first reaching age $r$ before absorption, the probability of at least one lead-changing transition during the age-$r$ stage is at least $r/d$.

Enumerate the deterministic age levels by
\[
        r_i:=s_0+i-1,\qquad 1\le i\le a-s_0.
\]
Let
\[
        R_i:=\{\sigma_{r_i}<\tau\}.
\]
On $R_i$, let
\[
        \eta_i:=\inf\{n>\sigma_{r_i}: S_n\ne r_i\}\wedge \tau
\]
be the end of the age-$r_i$ stage, with absorption included if it occurs before the age changes.  Let $I_i$ be the indicator of the event that some transition with pre-transition time
\[
        \sigma_{r_i}\le n<\eta_i
\]
changes the lead, i.e. has $|G_{n+1}-G_n|=1$.  This convention includes a transition into absorption if that transition occurs while the chain is at age $r_i$.

We spell out the filtration used for the domination step.  Let $\mathcal H_0$ be the initial sigma-field.  For each $i$, let $\mathcal H_i$ be the sigma-field generated by the stopped path up to the end of the age-$r_i$ stage, with a fixed cemetery completion on $R_i^c$.  Equivalently, on $R_i$ the sigma-field $\mathcal H_i$ contains the path up to time $\eta_i$, while on $R_i^c$ no further randomness is revealed.  Then $R_i\in\mathcal H_{i-1}$ and $I_i$ is $\mathcal H_i$-measurable.  Conditional on $\mathcal H_{i-1}$ and on $R_i$, the first changed transition made from age $r_i$ is lead-changing with probability at least $r_i/d$.  Therefore the probability that the whole age-$r_i$ stage contains at least one lead-changing transition is also at least $r_i/d$:
\[
        \Prob(I_i=1\mid \mathcal H_{i-1})\ge \frac{r_i}{d}
        \qquad\text{on }R_i.
\]

Apply Lemma~\ref{lem:bernoulli-domination} with the stage filtration
$\mathcal F_i=\mathcal H_i$, with $p_i=r_i/d$, and with
$A=\{\tau_a<\tau\}$.  On $A$, every age $r_i=s_0,\dots,a-1$ is reached before absorption, and the corresponding age-stage events are disjoint.  Therefore
\[
        L_a\ge \sum_{i=1}^{a-s_0} I_i
\]
on $A$.  Hence, for independent Bernoulli variables $(\xi_r)_{r=s_0}^{a-1}$ with
$\Prob(\xi_r=1)=r/d$ and
\[
        B_a:=\sum_{r=s_0}^{a-1}\xi_r,
\]
we have, for every $t\ge0$,
\[
        \Prob(\tau_a<\tau,\,L_a<t)
        \le
        \Prob\left(B_a<t\right).
\]

Its mean satisfies
\[
        \mu_a:=\E B_a=\frac{1}{d}\sum_{r=s_0}^{a-1}r
        \ge \frac{a^2}{4d},
\]
for $a\ge 4s_0$.  Chernoff's bound gives
\begin{equation}\label{eq:chernoff-L}
        \Prob(\tau_a<\tau,\,L_a<\mu_a/2)
        \le
        \exp(-c\mu_a).
\end{equation}

Now look only at the lead-changing transitions.  At an interior state $(s,g)$, conditional on a lead change, the probability of an upward lead change is
\[
        p_+(s,g)=\frac{(m-g)s}{(m-g)s+(s+g)m}.
\]
Since
\[
        (s+g)m-(m-g)s=gd\ge0,
\]
we have $p_+(s,g)\le1/2$.  The coupling is by common uniforms: at each lead-changing step, use the same uniform variable to decide whether the coupon-collector lead increases with probability $p_+(s,g)$ and whether an auxiliary simple symmetric walk increases with probability $1/2$.  Since $p_+(s,g)\le1/2$ at every interior state, the lead process at lead-changing times is pathwise dominated by the simple symmetric walk until absorption.  More explicitly, if $G^{\rm lc}_j$ is the coupon-collector lead after the $j$th lead-changing transition and $Y_j$ is the coupled simple symmetric walk, then $G^{\rm lc}_j\le Y_j$ for every $j$ before absorption.  Thus, whenever the coupon-collector path survives for at least $N$ lead-changing transitions, one must have
\[
        \min_{0\le j\le N}Y_j\ge0.
\]

For a simple symmetric random walk $Y$ started at $g_0$, the reflection principle gives the standard ballot estimate; see, for example, \cite[Ch.~III]{Feller}
\begin{equation}\label{eq:ssrw-survival}
        \Prob\left(\min_{0\le j\le N}Y_j\ge0\right)
        \le \frac{C_{g_0}}{\sqrt N}.
\end{equation}
Indeed, this follows by summing the reflected kernel
\[
        \Prob_{g_0}(Y_N=y)-\Prob_{-g_0-2}(Y_N=y),
        \qquad y\ge0.
\]
Combining \eqref{eq:chernoff-L} and \eqref{eq:ssrw-survival} with $N=\lfloor\mu_a/2\rfloor$ gives
\[
\begin{aligned}
\Prob_{s_0,g_0}(\tau_a<\tau)
&\le
\Prob_{s_0,g_0}(\tau_a<\tau,\,L_a<N)
+
\Prob_{s_0,g_0}(\tau_a<\tau,\,L_a\ge N)\\
&\le
\exp(-c\mu_a)+
\Prob\left(\min_{0\le j\le N}Y_j\ge0\right)\\
&\le
\exp(-c\mu_a)+\frac{C_{g_0}}{\sqrt{\mu_a}}
\le
C_{g_0}\frac{\sqrt d}{a}.
\end{aligned}
\]
This proves \eqref{eq:survival-bound}.
\end{proof}

\begin{remark}
Only the nontrivial range $a\ge4\max\{s_0,\lceil\sqrt d\rceil\}$ of Lemma~\ref{lem:survival} is used below. Ages below this threshold are handled separately by the trivial bound and the occupation estimate.
\end{remark}

\begin{lemma}[Occupation bound]\label{lem:occupation}
Let $I=[a,b)\cap\mathbb Z$, with $s_0\le a<b\le d$.  The endpoints $a,b$ need not be integers.  Then
\begin{equation}\label{eq:occupation-bound}
        \E_{s_0,g_0}\sum_{n<\tau}\Ind_{\{S_n\in I\}}
        \le
        2(b-a+1)\Prob_{s_0,g_0}(\tau_a<\tau).
\end{equation}
\end{lemma}

\begin{proof}
The coordinate $S_n$ is nondecreasing.  At a state $(s,g)$ with $m=d-s$, the probability that the next changed transition increases $S$ is
\[
        \frac{(m-g)m+(s+g)m}{D(s,g)}
        =\frac{md}{md+(m-g)s}.
\]
Since $m-g\le m$, this probability is at least
\[
        \frac{md}{md+ms}
        =\frac{d}{d+s}
        \ge \frac{1}{2}.
\]
Therefore, conditional on reaching a fixed age level $s$, the expected number of visits to that age before moving to $s+1$ or being absorbed is at most $2$.  Summing over the at most $b-a+1$ integer age levels in $I$ and multiplying by the probability of reaching age $a$ proves the claim.
\end{proof}

\begin{proof}[Proof of Lemma~\ref{lem:green}]
We split the state space into an initial range, a middle range, and a terminal range.  Throughout, constants may depend on the fixed initial lead $g_0$.

Let
\[
        A_0:=4\max\{s_0,\lceil\sqrt d\rceil\}.
\]
Since $s_0=o(d)$, for all sufficiently large $d$ we have $A_0\le d/4$.  All estimates in the rest of the proof are taken along this large-$d$ range.

\smallskip
\noindent\emph{Initial range.}
For $s_0\le S_n<A_0$, we have $m=d-S_n\ge 3d/4$, and therefore by \eqref{eq:Delta-simple-bound},
\[
        \Delta(S_n,G_n)\le \frac{4}{3d^2}.
\]
By Lemma~\ref{lem:occupation}, with the trivial survival bound, the expected number of visits to this range is $O(A_0)$.  Hence the initial contribution is
\[
        O\left(\frac{A_0}{d^2}\right)
        =O\left(\frac{s_0+\sqrt d}{d^2}\right).
\]
Since $H(s_0,g_0)=(g_0+1)/(d-s_0+1)\asymp_{g_0} d^{-1}$, this is
\begin{equation}\label{eq:initial-contrib}
        O_{g_0}\left(H(s_0,g_0)\left(\frac{s_0}{d}+\frac{1}{\sqrt d}\right)\right).
\end{equation}

\smallskip
\noindent\emph{Middle range.}
We next cover $A_0\le S_n<d/2$ by truncated dyadic intervals
\[
        I_j=[a_j,\min\{2a_j,d/2\})\cap\mathbb Z,
        \qquad
        a_j=2^jA_0,
\]
for those $j$ with $a_j<d/2$.  These intervals are disjoint and cover the whole middle range; the truncation handles the endpoint when $d/2$ is not an exact dyadic multiple of $A_0$.  On each $I_j$, $m\ge d/2$, so again
\[
        \Delta\le \frac{2}{d^2}.
\]
By Lemmas~\ref{lem:occupation} and~\ref{lem:survival}, using the nontrivial survival bound since $a_j\ge A_0\ge4\max\{s_0,\lceil\sqrt d\rceil\}$,
\[
        \E\sum_{n<\tau}\Ind_{\{S_n\in I_j\}}
        \le C(a_j+1)\frac{\sqrt d}{a_j}
        \le C\sqrt d.
\]
Thus each middle annulus contributes at most
\[
        \frac{C\sqrt d}{d^2}.
\]
There are $O(\log d)$ such annuli, so the middle contribution is
\begin{equation}\label{eq:middle-contrib}
        O\left(\frac{\log d}{d^{3/2}}\right)
        =O_{g_0}\left(H(s_0,g_0)\frac{\log d}{\sqrt d}\right).
\end{equation}

\smallskip
\noindent\emph{Terminal range.}
It remains to treat $S_n\ge d/2$.  First separate the final age $s=d-1$.  By \eqref{eq:Delta-simple-bound}, Lemma~\ref{lem:occupation}, and Lemma~\ref{lem:survival},
\[
        \E\sum_{n<\tau}\Delta(S_n,G_n)\Ind_{\{S_n=d-1\}}
        \le
        \frac{1}{2d}\,
        \E\sum_{n<\tau}\Ind_{\{S_n=d-1\}}
        \le
        \frac{C_{g_0}}{d}\Prob_{s_0,g_0}(\tau_{d-1}<\tau)
        =O_{g_0}(d^{-3/2}).
\]
For the remaining terminal ages, use dyadic intervals in the missing-coupon variable $m=d-s$.  For dyadic numbers $M_j=2^j$ with $1\le M_j\le d/2$, set
\[
        J_j:=\{s\in\mathbb Z:d/2\le s\le d-2,
        \ M_j\le d-s<2M_j\}.
\]
Empty intervals are ignored.  The nonempty intervals $J_j$ cover every integer $s$ with $d/2\le s\le d-2$: writing $m=d-s$, choose $M_j=2^{\lfloor\log_2 m\rfloor}$, so $M_j\le m<2M_j$.  On $J_j$, $m+1\ge M_j$, whence
\[
        \Delta\le \frac{1}{dM_j}.
\]
The length of $J_j$ is at most $M_j$.  To enter a nonempty $J_j$, the chain must reach its left endpoint
\[
        a_j:=\min J_j\ge d/2.
\]
For all sufficiently large $d$, this threshold lies in the nontrivial range of Lemma~\ref{lem:survival}, and hence
\[
        \Prob_{s_0,g_0}(\tau_{a_j}<\tau)
        \le C_{g_0}\frac{\sqrt d}{a_j}
        \le \frac{C_{g_0}}{\sqrt d}.
\]
Lemma~\ref{lem:occupation} therefore gives
\[
        \E\sum_{n<\tau}\Ind_{\{S_n\in J_j\}}
        \le \frac{C_{g_0}M_j}{\sqrt d}.
\]
Therefore the contribution of $J_j$ is at most
\[
        \frac{C_{g_0}}{d\sqrt d}.
\]
There are $O(\log d)$ terminal annuli, so the terminal contribution is
\begin{equation}\label{eq:terminal-contrib}
        O_{g_0}\left(\frac{\log d}{d^{3/2}}\right)
        =O_{g_0}\left(H(s_0,g_0)\frac{\log d}{\sqrt d}\right).
\end{equation}

Combining \eqref{eq:initial-contrib}, \eqref{eq:middle-contrib}, and \eqref{eq:terminal-contrib} proves \eqref{eq:green-estimate}.
\end{proof}

\section{Completion of the lead-excursion estimate}

We need one more negligible estimate: the probability of simultaneous terminal completion during the lead excursion.

\begin{lemma}[Terminal tie is negligible]\label{lem:terminal-tie}
Fix $g_0\ge0$.  If $0\le s_0<d-g_0$ and $s_0=o(d)$, then
\[
        \Prob_{s_0,g_0}\bigl((S_\tau,G_\tau)=(d,0)\bigr)
        =O_{g_0}(d^{-3/2})
        =o(H(s_0,g_0)).
\]
\end{lemma}

\begin{proof}
The only interior state from which the chain can enter the simultaneous terminal point $(d,0)$ is $(d-1,0)$.  Indeed, at age $d-1$ the safe region has only $g=0$ as an interior lead; $g=1$ is already the success boundary.

At $(d-1,0)$, the changed-state transition weights are
\[
        1,\quad d-1,\quad d-1,
\]
corresponding respectively to simultaneous completion, $A$ alone completing, and $B$ alone completing.  Hence the conditional probability of entering $(d,0)$ from $(d-1,0)$ is
\[
        \frac{1}{2d-1}.
\]
Therefore
\[
        \Prob_{s_0,g_0}\bigl((S_\tau,G_\tau)=(d,0)\bigr)
        \le \frac{1}{2d-1}\Prob_{s_0,g_0}(\tau_{d-1}<\tau).
\]
For all sufficiently large $d$, the threshold $d-1$ lies in the nontrivial range of Lemma~\ref{lem:survival}; hence
\[
        \Prob_{s_0,g_0}(\tau_{d-1}<\tau)
        \le C_{g_0}\frac{\sqrt d}{d-1}.
\]
Combining the last two displays gives the asserted $O_{g_0}(d^{-3/2})$ bound.  Since $s_0=o(d)$ and $g_0$ is fixed,
\[
        H(s_0,g_0)=\frac{g_0+1}{d-s_0+1}\asymp_{g_0} d^{-1},
\]
so the terminal-tie probability is $o(H(s_0,g_0))$.
\end{proof}

\begin{proof}[Proof of Theorem~\ref{thm:lead-excursion}]
The process
\[
        M_n:=H(S_{n\wedge\tau},G_{n\wedge\tau})
        -\sum_{r<n\wedge\tau}\Delta(S_r,G_r)
\]
is a martingale.  Since Lemma~\ref{lem:self-loop-erasure} gives $\tau\le 2d$ almost surely, optional stopping applies directly at $\tau$ and yields
\begin{equation}\label{eq:optional-H}
        \E_{s_0,g_0}H(S_\tau,G_\tau)
        =H(s_0,g_0)+
        \E_{s_0,g_0}\sum_{n<\tau}\Delta(S_n,G_n).
\end{equation}
At the success boundary, $H=1$.  At the unsafe boundary, $H=0$.  At the simultaneous terminal point $(d,0)$, $H=1$, but that event does not count as a win for $A$.  Therefore
\[
        \E_{s_0,g_0}H(S_\tau,G_\tau)
        =U_d(s_0,g_0)+
        \Prob_{s_0,g_0}\bigl((S_\tau,G_\tau)=(d,0)\bigr).
\]
Using Lemmas~\ref{lem:green} and~\ref{lem:terminal-tie} in \eqref{eq:optional-H}, we obtain
\[
        U_d(s_0,g_0)=H(s_0,g_0)
        \left(1+O_{g_0}\left(\frac{s_0}{d}+\frac{\log d}{\sqrt d}\right)\right),
\]
as claimed.
\end{proof}

\section{Proof of the main theorem}

We now combine the tie-skeleton formula with the lead-excursion estimate.

\begin{proof}[Proof of Theorem~\ref{thm:main}]
Recall from \eqref{eq:tie-expectation} that
\[
        E_d=\frac{1}{2}\sum_{k=1}^{d-1}\pi_{d,k}w_{d-k},
\]
where $w_{d-k}=U_d(k,1)$ in $(s,g)$ coordinates.

Choose a fixed constant $R$ so large that the tail bound in \eqref{eq:first-break-tail} is $o(d^{-1})$, and set
\[
        B_d:=R\sqrt{d\log d}.
\]
Then, by Lemma~\ref{lem:first-break},
\begin{equation}\label{eq:main-tail-small}
        \sum_{k>B_d}\pi_{d,k}=o(d^{-1}).
\end{equation}
Since $0\le w_{d-k}\le1$, the contribution of $k>B_d$ to \eqref{eq:tie-expectation} is $o(d^{-1})$.

On the complementary range $1\le k\le B_d$, we have $k=o(d)$ uniformly.  Therefore Theorem~\ref{thm:lead-excursion}, applied with $s_0=k$ and $g_0=1$, gives
\[
        w_{d-k}=U_d(k,1)
        =\frac{2}{d-k+1}
        \left(1+O\left(\frac{B_d}{d}+\frac{\log d}{\sqrt d}\right)\right)
        =\frac{2}{d}(1+o(1))
\]
uniformly for $1\le k\le B_d$.  Hence
\[
        E_d
        =\frac{1}{2}\sum_{k\le B_d}\pi_{d,k}\frac{2}{d}(1+o(1))
        +o(d^{-1}).
\]
The probability of no one-sided break before simultaneous completion is $O(e^{-cd})$, again by Lemma~\ref{lem:first-break}; hence
\[
        \sum_{k=1}^{d-1}\pi_{d,k}=1+o(1).
\]
Together with \eqref{eq:main-tail-small}, this implies
\[
        \sum_{k\le B_d}\pi_{d,k}=1+o(1).
\]
Consequently
\[
        E_d\sim \frac{1}{d}.
\]
Finally, $b_d=2E_d$, so
\[
        b_d\sim \frac{2}{d}.
\]
\end{proof}

\begin{corollary}[Conditional first-break law and moments]
\label{cor:conditional-first-break}
Let $K_d$ be the extended first-break level from
Lemma~\ref{lem:first-break}, and let
\[
        \mathcal B_d:=\{\text{the ultimate winner was never behind}\}.
\]
Then, conditional on $\mathcal E_A$,
\[
        \frac{K_d}{\sqrt d}\Rightarrow X,
        \qquad
        \Prob(X\in \dd x)=2x e^{-x^2}\Ind_{\{x>0\}}\dd x .
\]
The same convergence holds conditional on $\mathcal B_d$.

Moreover, for every fixed $p>0$,
\[
        \E\left[\left(\frac{K_d}{\sqrt d}\right)^p\right]
        \longrightarrow
        \Gamma\left(1+\frac p2\right),
\]
and
\[
        \E\left[
        \left(\frac{K_d}{\sqrt d}\right)^p
        \,\middle|\,\mathcal E_A\right]
        \longrightarrow
        \Gamma\left(1+\frac p2\right).
\]
The same conditional moment convergence holds with $\mathcal E_A$ replaced by
$\mathcal B_d$.  In particular,
\[
        \E K_d\sim \frac{\sqrt\pi}{2}\sqrt d,
        \qquad
        \operatorname{Var}(K_d)\sim
        \left(1-\frac\pi4\right)d,
\]
and the same expectation and variance asymptotics hold conditional on
$\mathcal E_A$ and conditional on $\mathcal B_d$.
\end{corollary}

\begin{proof}
The unconditional weak convergence is Lemma~\ref{lem:first-break}.  We first
upgrade it to convergence of fixed moments.  The Rayleigh tail bound
\eqref{eq:first-break-rayleigh-tail}, applied on dyadic annuli, implies
that for every fixed $p>0$ there are constants $C_p,c_p>0$ such that, for all
$M\ge1$ and all sufficiently large $d$,
\begin{equation}\label{eq:moment-tail-unconditional}
        \sum_{k\ge M\sqrt d}\left(\frac{k}{\sqrt d}\right)^p\pi_{d,k}
        \le C_p e^{-c_pM^2}+O(d^{p/2}e^{-cd}).
\end{equation}
The exceptional atom at $K_d=d$ has mass $O(e^{-cd})$, so its contribution to
the $p$th scaled moment is also $O(d^{p/2}e^{-cd})$.  Thus
$(K_d/\sqrt d)^p$ is uniformly integrable.  Since
\[
        \int_0^\infty x^p\,2xe^{-x^2}\,\dd x
        =\Gamma\left(1+\frac p2\right),
\]
the unconditional moment convergence follows.

It remains to prove the conditional statements.  From the tie-skeleton formula,
\[
        \Prob(K_d=k,\mathcal E_A)
        =\frac12\pi_{d,k}w_{d-k},
        \qquad 1\le k\le d-1,
\]
where $w_{d-k}=U_d(k,1)$.  Let
\[
        L_d:=R\sqrt{d\log d},
\]
where $R$ is a fixed constant, chosen large enough when needed.  Uniformly for
$1\le k\le L_d$, Theorem~\ref{thm:lead-excursion} gives
\[
        w_{d-k}
        =\frac{2}{d-k+1}
        \left(1+O\left(\frac{L_d}{d}+\frac{\log d}{\sqrt d}\right)\right)
        =\frac{2}{d}(1+o(1)).
\]
Also $E_d\sim d^{-1}$.  Therefore, for every bounded continuous function $f$,
\begin{align*}
        \E\left[
        f\left(\frac{K_d}{\sqrt d}\right)
        \,\middle|\,\mathcal E_A\right]
        &=
        \frac{1}{E_d}\sum_{k=1}^{d-1}
        f\left(\frac{k}{\sqrt d}\right)
        \frac12\pi_{d,k}w_{d-k}  \\
        &=
        \sum_{k\le L_d}
        f\left(\frac{k}{\sqrt d}\right)\pi_{d,k}+o(1).
\end{align*}
The contribution from $k>L_d$ is $o(1)$ by \eqref{eq:first-break-tail}, after
choosing $R$ large enough, because $0\le w_{d-k}\le1$ and $E_d\sim d^{-1}$.
Lemma~\ref{lem:first-break} now gives
\[
        \E\left[
        f\left(\frac{K_d}{\sqrt d}\right)
        \,\middle|\,\mathcal E_A\right]
        \longrightarrow
        \int_0^\infty f(x)2xe^{-x^2}\,\dd x .
\]
This proves the conditional Rayleigh law.

We next prove conditional moment convergence.  Fix $p>0$ and choose $R$ so
large that
\[
        d^{1+p/2}\sum_{k>L_d}\pi_{d,k}\longrightarrow0,
\]
which is possible by \eqref{eq:first-break-tail}.  Since $0\le w_{d-k}\le1$,
$E_d\sim d^{-1}$, and $(k/\sqrt d)^p\le d^{p/2}$, the contribution of
$k>L_d$ to
\[
        \E\left[\left(\frac{K_d}{\sqrt d}\right)^p
        \,\middle|\,\mathcal E_A\right]
\]
is $o(1)$.  On $k\le L_d$ the weight $w_{d-k}$ is $2d^{-1}(1+o(1))$
uniformly.  Hence, for each fixed $M\ge1$,
\[
\begin{aligned}
        \E\left[\left(\frac{K_d}{\sqrt d}\right)^p
        \Ind_{\{K_d\le M\sqrt d\}}
        \,\middle|\,\mathcal E_A\right]
        &=
        \sum_{k\le M\sqrt d}\left(\frac{k}{\sqrt d}\right)^p\pi_{d,k}+o(1).
\end{aligned}
\]
Letting $d\to\infty$ and then $M\to\infty$, and using
\eqref{eq:moment-tail-unconditional} for the remaining intermediate tail
$M\sqrt d<k\le L_d$, gives
\[
        \E\left[\left(\frac{K_d}{\sqrt d}\right)^p
        \,\middle|\,\mathcal E_A\right]
        \longrightarrow
        \Gamma\left(1+\frac p2\right).
\]

Finally,
\[
        \mathcal B_d=\mathcal E_A\sqcup\mathcal E_B,
\]
and the two events are exchanged by symmetry.  The conditional law of $K_d$
given $\mathcal E_B$ is therefore the same as its conditional law given
$\mathcal E_A$.  Hence the same conclusions hold conditional on
$\mathcal B_d$.

The displayed expectation and variance asymptotics follow from the moment
statements with $p=1$ and $p=2$, since
\[
        \E X=\frac{\sqrt\pi}{2},
        \qquad
        \E X^2=1.
\]
\end{proof}

\section{Numerical check from the exact recursion}

The proof above is asymptotic, but the finite recursion \eqref{eq:U-recursion} also gives a simple deterministic numerical check.  Evaluating \eqref{eq:U-recursion} backwards over the triangular state space gives the following values; no simulation is involved.  For reproducibility: impose the displayed boundary values, fill the triangular array in decreasing $s$ and decreasing $g$ using \eqref{eq:U-recursion}, and return $b_d=2U_d(0,0)$.
\[
\begin{array}{c|c|c}
 d & b_d & d b_d \\
\hline
 20   & 0.1534023902 & 3.068047804 \\
 50   & 0.0571300559 & 2.856502794 \\
 100  & 0.0268231002 & 2.682310024 \\
 200  & 0.0126463314 & 2.529266273 \\
 500  & 0.0047395234 & 2.369761723 \\
 1000 & 0.0022790277 & 2.279027674 \\
 2000 & 0.0011046906 & 2.209381159
\end{array}
\]
The convergence is slow, as expected from the proof, and is consistent with the available relative error scale $O((\log d)/\sqrt d)$.  The values are nevertheless consistent with the limit $d b_d\to2$.

\section{Remarks on the flux mechanism}

In the terminology introduced in the introduction, the proof is a terminal-condition flux calculation in three layers.

\begin{enumerate}[label=(\roman*)]
\item The tie boundary $g=0$ is not treated as an ordinary interior boundary.  It is a renewal set.

\item The first one-sided tie-break has explicit incoming flux
\[
        \pi_{d,k}
        =
        \left(\prod_{r=1}^{k-1}\frac{d-r}{d+r}\right)\frac{2k}{d+k},
\]
which, after scaling by $d^{1/2}$, converges to the Rayleigh-type density $2x e^{-x^2}$ on $x>0$.

\item After the tie break, the lead excursion is governed to first order by the Catalan comparison harmonic
\[
        H(s,g)=\frac{g+1}{d-s+1}.
\]
The exact simultaneous-round model makes $H$ slightly subharmonic in this sense; the dyadic Green estimate proves that the accumulated defect is negligible uniformly throughout the early first-break regime.
\end{enumerate}

Thus Myers--Wilf's finite decomposition into tails, frames, and ribbons is replaced asymptotically by a renewal decomposition at the tie set followed by a Catalan lead-excursion calculation.  The word ``flux'' is meant only to emphasize that the leading contribution is the entrance mass through the first one-sided tie-break boundary, not to add a separate assumption to the probabilistic proof.

\section{Further directions}

The proof suggests a more general transfer principle for ballot-type
competition problems.  In the present model the calculation separates into
two pieces: an entrance law through the tie boundary, and a post-entrance
survival probability in a shrinking wedge.  The first piece is the renewal
law of the first one-sided tie-break; the second is controlled by the Catalan
comparison harmonic, with a Green estimate showing that the defect of this
harmonic in the exact chain is negligible.

A natural general problem is to formulate a renewal--Catalan transfer theorem
for triangular families of absorbing Markov chains in wedges.  Such a theorem
would assume that the tie boundary is a renewal set, that the first one-sided
entrance occurs at a scale $a_d=o(d)$ with a tight entrance lead, that an
appropriate Catalan or gambler's-ruin harmonic has Green-negligible defect,
and that terminal ties are negligible.  Under these hypotheses, the ballot
probability should be given asymptotically by the entrance flux averaged
against the comparison harmonic.  In the present problem, since $b_d=2E_d$,
this principle reduces to
\[
        b_d \sim
        \sum_k \pi_{d,k}\frac{2}{d-k+1}
        \sim \frac{2}{d}.
\]
The proof above may be viewed as one concrete instance of this transfer
principle.  It would also be natural to seek second-order asymptotics for
$b_d$, since the next terms should separate the Rayleigh entrance correction
from the accumulated Green defect of the Catalan comparison harmonic.

Several further limit questions are suggested by this decomposition.  The
conditional first-break law identifies the entrance scale of successful paths,
but not their subsequent shape.  A natural next problem is to prove a path-level
version of the renewal--Catalan principle: after the first one-sided tie-break
and conditional on eventual ballot success, the lead process should be
approximated by a Doob transform associated with the Catalan comparison
harmonic.  A related terminal question is to identify the law of the loser's
residual number of missing coupons when the winner completes.  These questions
would require stronger path-space estimates than the Green estimate used here,
which controls hitting probabilities but not the full conditioned trajectory.

Another direction is to extend the argument to non-uniform coupon
probabilities.  In that setting the pair of collection counts no longer forms
a closed Markov chain, so the entrance law should depend on residual coupon
weights rather than only on the number of collected types.  A successful
extension would require replacing the scalar age coordinate by an appropriate
hazard or residual-mass coordinate, and then proving an analogue of the
Green-negligibility estimate for the corresponding comparison harmonic.


\begin{thebibliography}{9}

\bibitem{Feller}
William Feller,
\emph{An Introduction to Probability Theory and Its Applications, Vol. I},
3rd ed., John Wiley \& Sons, New York, 1968.

\bibitem{GV}
Ira Gessel and G\'erard Viennot,
Binomial determinants, paths, and hook length formulae,
\emph{Advances in Mathematics} \textbf{58} (1985), no.~3, 300--321.
DOI: \href{https://doi.org/10.1016/0001-8708(85)90121-5}{10.1016/0001-8708(85)90121-5}.

\bibitem{MW}
Amy N. Myers and Herbert S. Wilf,
Some new aspects of the coupon-collector's problem,
\emph{SIAM Journal on Discrete Mathematics} \textbf{17} (2003), no.~1, 1--17.
DOI: \href{https://doi.org/10.1137/S0895480102403076}{10.1137/S0895480102403076}; arXiv:math/0304229.

\end{thebibliography}
\end{document}